\documentclass[12pt]{amsart}
\usepackage[T1]{fontenc}
\usepackage{amsmath,amssymb,amsthm}
\usepackage[margin=2.9cm]{geometry}
\usepackage{xcolor}
\usepackage[colorlinks=true,linkcolor=blue,citecolor=blue]{hyperref}
\newtheorem{theorem}{Theorem}[section]
\newtheorem{proposition}[theorem]{Proposition}
\newtheorem{lemma}[theorem]{Lemma}
\newtheorem{corollary}[theorem]{Corollary}

\theoremstyle{remark}
\newtheorem{remark}[theorem]{Remark}

\theoremstyle{definition}
\newtheorem{definition}[theorem]{Definition}

\theoremstyle{plain}
\newtheorem{introtheorem}{Theorem}

\newtheorem{introcorollary}[introtheorem]{Corollary}

\DeclareMathOperator{\Ghd}{Ghd}
\DeclareMathOperator{\Gcd}{Gcd}
\DeclareMathOperator{\Gfd}{Gfd}
\DeclareMathOperator{\Gpd}{Gpd}
\DeclareMathOperator{\fd}{fd}
\DeclareMathOperator{\pd}{pd}
\DeclareMathOperator{\cd}{cd}
\DeclareMathOperator{\hd}{hd}
\DeclareMathOperator{\sfli}{sfli}
\DeclareMathOperator{\spli}{spli}
\DeclareMathOperator{\silp}{silp}
\DeclareMathOperator{\id}{id}
\DeclareMathOperator{\Gwgl}{Gwgl.dim}
\DeclareMathOperator{\Ggl}{Ggl.dim}
\DeclareMathOperator{\wgl}{w.gl.dim}
\DeclareMathOperator{\gldim}{gl.dim}
\DeclareMathOperator{\Hom}{Hom}
\DeclareMathOperator{\Ext}{Ext}
\DeclareMathOperator{\Tor}{Tor}
\DeclareMathOperator{\im}{im}
\DeclareMathOperator{\coker}{coker}
\newcommand{\PGFdim}{\operatorname{PGF\text{-}dim}}
\newcommand{\PGFcd}{\widetilde{\Gcd}}
\newcommand{\GFlat}{\mathtt{GFlat}}
\newcommand{\GProj}{\mathtt{GProj}}
\newcommand{\PGF}{\mathtt{PGF}}
\newcommand{\Z}{\mathbb{Z}}
\newcommand{\Q}{\mathbb{Q}}
\newcommand{\alephm}{\aleph_m}

\begin{document}

\title[$\alephm$-presented Gorenstein flat modules and Gorenstein dimensions of groups]{$\alephm$-presented Gorenstein flat modules and\\ Gorenstein dimensions of groups}
\author{Dimitra-Dionysia Stergiopoulou}
\address{Department of Mathematics, University of Thessaly, Lamia, Greece}
\email{dstergiop@math.uoa.gr, dstergiopoulou@uth.gr}

\begin{abstract}
Let $G$ be a group and $k$ be a commutative ring. The Gorenstein homological dimension $\Ghd_kG$ and the Gorenstein cohomological dimension $\Gcd_kG$ of  $G$ over $k$ are defined as the Gorenstein flat and the Gorenstein projective dimension, respectively, of the trivial $kG$-module $k$. We prove that the Gorenstein cohomological dimension of a countable group $G$ is at most one more than its Gorenstein homological dimension and, more generally, that $\Gcd_kG\le\Ghd_kG+m+1$ for every $m\ge0$ and every group $G$ of cardinality at most $\aleph_m$. When the supremum $\sfli k$ of the flat dimensions of the injective $k$-modules is finite, this gives $\Ghd_kG\le\Gcd_kG\le\Ghd_kG+1$ for every countable group $G$, so that the two dimensions are finite simultaneously. These results are Gorenstein analogues of Bieri's inequality for the homological and the cohomological dimension of countable groups. They rest on a module-theoretic theorem of independent interest, which is a Gorenstein version of results of Jensen and Osofsky on the projective dimension of flat modules: over an arbitrary ring, every $\aleph_m$-presented Gorenstein flat module has projectively coresolved Gorenstein flat dimension at most $m+1$. An appendix characterizes the modules of PGF-dimension at most $n$ as the direct summands of the strongly $n$-PGF modules. 
\end{abstract}

\subjclass[2020]{Primary 20J05; Secondary 16E10, 18G25, 20J06}
\keywords{$\aleph_m$-presented module, Gorenstein flat module, PGF module, PGF-dimension, Gorenstein homological dimension, countable group, projective dimension of flat modules}
\renewcommand{\subjclassname}{%
  \textup{2020} Mathematics Subject Classification}
\maketitle

\section{Introduction}
For a group $G$ and a commutative ring $k$, the homological dimension $\hd_kG$ never exceeds the cohomological dimension $\cd_kG$, and the two may differ. When the two dimensions coincide and how far apart they can be are classical questions in the homological theory of groups. They coincide for groups of type $\mathrm{FP}_\infty$ over $k$, in particular for duality groups, whereas for countable groups the cohomological dimension is controlled by the homological one up to an error of at most one: Bieri proved that $\cd_kG\le\hd_kG+1$ for every countable group $G$ (see \cite[Theorem~4.6]{Bieri}). At the level of modules, this rests on the fact that a countably presented flat module has projective dimension at most $1$, and the theorems of Jensen and Osofsky bounding the projective dimension of an $\alephm$-presented flat module by $m+1$ \cite{Jen,Oso} yield in the same way the inequality $\cd_kG\le\hd_kG+m+1$ whenever $|G|\le\alephm$ (see the remark in \cite[p.~56]{Bieri}). Thus, in the classical setting, the question whether the finiteness of $\hd_kG$ implies that of $\cd_kG$ is settled for groups of cardinality less than $\aleph_\omega$, while it remains open in general.

Torsion can force the classical homological and cohomological dimensions to be infinite. More precisely, if $G$ contains an element of finite order $n$ such that $n$ is not invertible in $k$, then $\operatorname{hd}_k G=\operatorname{cd}_k G=\infty$. Their Gorenstein counterparts, $\Ghd_k G$ and $\Gcd_k G$, refine these invariants and can remain finite even in the presence of such torsion.
 They are defined as the Gorenstein flat and the Gorenstein projective dimension, respectively, of the trivial $kG$-module $k$ (see \cite{ABHS,KS,RY} and \cite{BDT,ET}). They agree with the classical dimensions whenever the latter are finite, but they are finite for many groups with torsion. For instance, $\Ghd_kG=0$ for every locally finite group $G$ \cite[Theorem~1.4]{KS}, and we shall see that $\Gcd_kG\le m+1$ for every locally finite group of cardinality at most $\alephm$ (see Corollary~\ref{cor5.1}). The Gorenstein cohomological dimension is also connected to the classifying space $\underline{E}G$ for proper actions (see \cite{BDT}). One is thus led to ask which of the classical relations between $\hd_kG$ and $\cd_kG$ survive in the Gorenstein setting. It is known that $\Ghd_kG\le\Gcd_kG$ whenever $\sfli k<\infty$ (see \cite[Theorem~3.10]{KS}), with equality for groups of type $\mathrm{FP}_\infty$ over $k$ (see \cite[Theorem~B]{St}). The finiteness of the Gorenstein dimensions of groups has been studied in \cite{ET,KS}.

 The starting point of the present paper is the question whether the finiteness of $\Ghd_kG$ forces that of $\Gcd_kG$ and, if it does, by how much the two invariants can differ. We establish the Gorenstein analogue of Bieri's inequality: for every commutative ring $k$, the Gorenstein cohomological dimension of a countable group is at most one more than its Gorenstein homological dimension, and at most $m+1$ more for groups of cardinality at most $\alephm$. As in the classical case, the group-theoretic statement rests on a module-theoretic one, in which the projectively coresolved Gorenstein flat modules of \v{S}aroch--\v{S}\v{t}ov\'i\v{c}ek \cite{SS} play the role of the projective modules. We fix our notation before stating the results.

Throughout, rings are associative with unit and modules are left modules unless stated otherwise. For a group $G$ and a commutative ring $k$ we regard $k$ as a trivial module over the group algebra $kG$ and set $\Ghd_kG=\Gfd_{kG}k$,
$\Gcd_kG=\Gpd_{kG}k$, $\PGFcd_kG=\PGFdim_{kG}k$,
the Gorenstein homological dimension, the Gorenstein cohomological dimension, and the projectively coresolved Gorenstein flat dimension of $G$ over $k$, respectively (see \cite{St,KS}). We recall that a module is \emph{projectively coresolved Gorenstein flat} (PGF for short) if it is a cycle of an acyclic complex of projective modules which remains acyclic upon tensoring with any injective right module (see \cite{SS}). PGF modules are both Gorenstein projective and Gorenstein flat. The PGF-dimension $\PGFdim_RM$ of a module $M$ is defined by means of resolutions by PGF modules; it dominates both the Gorenstein projective and the Gorenstein flat dimension of $M$.

   Our first result is the Gorenstein analogue of the Jensen--Osofsky bounds, with the class of PGF modules in the role of the projectives.

\begin{introtheorem}\label{theoA}
Let $R$ be an arbitrary ring and $m$ be a nonnegative integer. Then, every $\alephm$-presented Gorenstein flat $R$-module $M$ satisfies
\[
\PGFdim_RM\;\le\;m+1 .
\]
In particular, every countably presented Gorenstein flat module has PGF-dimension at most $1$, and hence Gorenstein projective dimension at most $1$.
\end{introtheorem}

The bound in Theorem~\ref{theoA} is best possible for every $m$, already for flat modules (see Remark~\ref{rem4.2}). For rings of cardinality at most $\alephm$ the same bound holds for all Gorenstein flat modules, as a consequence of a theorem of Simson \cite{Simson} and Gruson--Jensen \cite{GJ} on the projective dimension of flat modules and of \cite[Proposition~15]{DE}; this yields an alternative proof of Corollary~\ref{cor4.4} for such rings (see Remark~\ref{rem4.6}). The point of Theorem~\ref{theoA} is that the hypothesis concerns the module and not the ring, and it is precisely this feature which makes its group-theoretic consequences independent of the coefficient ring. Theorem~\ref{theoA} is proved in Section~\ref{sec:cover}; the argument combines an approximation step (see Section~\ref{sec:step}), which produces a coresolution of $M$ by $\alephm$-presented flat modules, with a lemma on complexes (see Section~\ref{sec:cover}) which converts a bound on the projective dimension of the \emph{terms} of an $F$-totally acyclic complex into a bound on the PGF-dimension of its \emph{syzygies}.

Since the PGF-dimension dominates both the Gorenstein projective and the Gorenstein flat dimension (see Proposition~\ref{prop2.3}), the bound of Theorem~\ref{theoA} propagates, by dimension shifting, to every module admitting a projective resolution with $\alephm$-generated terms (see Corollary~\ref{cor4.3}), and in particular to all $\alephm$-generated modules over a ring all of whose left ideals are $\alephm$-generated, e.g.\ over a ring of cardinality at most $\alephm$ (see Corollary~\ref{cor4.4}).

Applied to the syzygies of the bar resolution, Theorem~\ref{theoA} yields the following statements.

\begin{introtheorem}\label{theoB}
Let $m$ be a nonnegative integer, $k$ be a commutative ring and $G$ be a group of cardinality at most $\alephm$. Then,
\[
\Gcd_kG\;\le\;\PGFcd_kG\;\le\;\Ghd_kG+m+1 .
\]
In particular, every countable group $G$ satisfies $\Gcd_kG\le\PGFcd_kG\le \Ghd_kG+1$.
\end{introtheorem}

\begin{introcorollary}\label{corC}
Let $k$ be a commutative ring with $\sfli k<\infty$ and let $G$ be a countable group. Then,
\[
\Ghd_kG\le\Gcd_kG\le\Ghd_kG+1.
\]
In particular, $\Ghd_kG<\infty$ if and only if $\Gcd_kG<\infty$.
\end{introcorollary}

The first inequality of Corollary~\ref{corC} is \cite[Theorem~3.10]{KS}. Both bounds of the corollary are attained: groups of type $\mathrm{FP}_\infty$ over $k$ with $\sfli k<\infty$ realize the lower one (see \cite[Theorem~B]{St}), while, for every such $k$ and every $n\ge0$, there are countable groups $G$ with $\Ghd_kG=n$ and $\Gcd_kG=n+1$, the simplest being the countably infinite locally finite groups for $n=0$ (see Remark~\ref{rem5.2} and Proposition~\ref{prop5.3}). Thus, for a countable group, the Gorenstein cohomological dimension is determined by the Gorenstein homological dimension up to an error of at most one, and the error does occur. More generally, the bound of Theorem~\ref{theoB} is attained for every $m\ge0$ already for $k=\Z$ and locally finite abelian groups (see Proposition~\ref{prop5.4}).

Two features of Theorem~\ref{theoB} are worth emphasizing. Firstly, the conclusion is obtained for the finer invariant $\PGFcd_kG$, which dominates $\Gcd_kG$; secondly, the cardinality hypothesis concerns only the group $G$. Theorem~\ref{theoB} thus shows that, for groups of cardinality less than $\aleph_\omega$, the finiteness of $\Ghd_kG$ implies that of $\Gcd_kG$ for every commutative ring $k$. The two invariants $\Gcd_kG$ and $\PGFcd_kG$ coincide as soon as $\sfli k<\infty$ (see \cite[Proposition 2.13]{St0}), so that under this mild hypothesis the bound of Theorem~\ref{theoB} is a statement about the Gorenstein cohomological dimension itself. Moreover, whenever $|G|\le\alephm$ and $\Ghd_kG<\infty$, Theorem~\ref{theoB} implies that $\PGFcd_kG<\infty$.
Consequently, \cite[Corollary~13(ii)]{DE} yields
$\Gcd_kG=\PGFcd_kG$.
Thus, for groups satisfying the cardinality hypothesis of
Theorem~\ref{theoB} and having finite Gorenstein homological dimension,
the two invariants coincide over every commutative
coefficient ring. 

Theorem~\ref{theoA} shows that the Gorenstein flat modules of small presentation are close to PGF modules, and hence to Gorenstein projective modules, whereas it is a long-standing open problem whether every Gorenstein projective module is Gorenstein flat. Moreover,
passing from modules to rings, Theorem~\ref{theoA} also yields a Gorenstein version of a classical theorem of Osofsky and Jensen on global dimensions: if every left ideal of a ring $R$ is $\alephm$-generated, for instance if $|R|\le\alephm$, then
\[
\Ggl R\;\le\;\Gwgl R+m+1
\]
(see Corollary~\ref{cor4.5}). Here $\Ggl R$ and $\Gwgl R$ denote the Gorenstein global and the Gorenstein weak global dimension of $R$. For a group algebra $kG$ with $|G|\le\alephm$ and all ideals of $k$ being $\alephm$-generated, as is the case if $|k|\le\alephm$, this gives $\Ggl(kG)\le\sfli(kG)+m+1$ and, in combination with the estimate of $\sfli(kG)$ obtained in \cite{KS},
\[
\silp(kG)\;\le\;\spli(kG)\;\le\;\sfli(kG)+m+1\;\le\;\Ghd_kG+\sfli k+m+1
\]
(see Corollaries~\ref{cor5.6} and~\ref{cor5.7}), a bound for the Gedrich--Gruenberg invariants of the group algebra in terms of the Gorenstein homological dimension of the group.

The paper is organized as follows. Section~\ref{sec:prelim} collects all necessary background in Gorenstein homological algebra and two cardinality lemmas on $\kappa$-presented modules. Section~\ref{sec:step} contains the approximation step (see Proposition~\ref{prop3.1}): an $\alephm$-presented Gorenstein flat module embeds, with $\alephm$-presented Gorenstein flat cokernel, in an $\alephm$-presented flat module, the embedding remaining injective after tensoring with any injective right module. Iterating, one obtains a coresolution by $\alephm$-presented flat modules which, spliced with a projective resolution, gives an $F$-totally acyclic complex. In Section~\ref{sec:cover}, we establish our main technical tool, Lemma~\ref{lem4.1}, which concerns $F$-totally acyclic complexes whose terms have bounded projective dimension. We prove Theorem~\ref{theoA} and discuss its sharpness in Remark~\ref{rem4.2}. We also derive consequences for modules admitting projective resolutions with $\alephm$-generated terms
(see Corollaries~\ref{cor4.3} and~\ref{cor4.4})
and obtain a Gorenstein version of Osofsky's inequality
(see Corollary~\ref{cor4.5}). Section~\ref{sec:groups} is devoted to groups: it contains the proofs of Theorem~\ref{theoB} and Corollary~\ref{corC}, the case of locally finite groups (see Corollary~\ref{cor5.1}), examples showing that the bounds obtained are attained (see Remark~\ref{rem5.2} and Propositions~\ref{prop5.3} and~\ref{prop5.4}), and applications to the invariants $\Ggl(kG)$, $\spli(kG)$ and $\silp(kG)$ of group algebras (see Corollaries~\ref{cor5.6} and \ref{cor5.7}). It also contains a criterion for the simultaneous finiteness of the Gorenstein dimensions of $G$ and of the Gedrich--Gruenberg invariants of $kG$ (see Corollary~\ref{cor5.8}). Finally, Appendix~\ref{sec:appendix} contains a characterization of the modules of PGF-dimension at most $n$ as the direct summands of the \emph{strongly $n$-PGF} modules (see Theorem~\ref{theoA.5}); it extends the description of the PGF modules as direct summands of strongly PGF modules given in \cite{SS}, provides the input for Lemma~\ref{lem4.1}, and is of independent interest.

\section{Preliminaries}\label{sec:prelim}

Throughout this section $R$ denotes an arbitrary associative ring with unit. We recall here the notions of Gorenstein homological algebra that will be used freely in the sequel, and we assemble in Proposition~\ref{prop2.3} the properties of the PGF-dimension on which the arguments of Sections~\ref{sec:step}--\ref{sec:groups} rest. Moreover, we prove in Subsection~\ref{subsec:cardinal} two cardinality lemmas which will be used throughout the paper.

\subsection{Gorenstein projective and Gorenstein flat modules}

Call an acyclic complex $\mathbf P$ of projective $R$-modules \emph{totally acyclic} when the complex $\Hom_R(\mathbf P,Q)$ is again acyclic for every projective $R$-module $Q$. A module is \emph{Gorenstein projective} if it can be realized as a cycle of some totally acyclic complex of projectives; we write $\GProj(R)$ for the class of all such modules. Replacing ``projective'' by ``Gorenstein projective'' in the definition of the projective dimension produces the \emph{Gorenstein projective dimension} $\Gpd_RM$ of a module $M$: it is the smallest integer $n\ge0$ for which some exact sequence
\[
0\rightarrow G_n\rightarrow G_{n-1}\rightarrow\cdots\rightarrow G_0\rightarrow M\rightarrow0
\]
with all $G_i\in\GProj(R)$ exists, and $\Gpd_RM=\infty$ when there is no such $n$. We refer to \cite{EJ,Hol} for a systematic treatment.

The flat side of the theory is obtained by exchanging the functors $\Hom_R(-,Q)$ for the functors $I\otimes_R-$, with $I$ ranging over the injective right $R$-modules. Accordingly, an acyclic complex $\mathbf F$ of flat $R$-modules is said to be \emph{totally acyclic in the flat sense}, or simply \emph{$F$-totally acyclic}, if $I\otimes_R\mathbf F$ is acyclic for every injective right $R$-module $I$; a module is \emph{Gorenstein flat} when it occurs as a cycle of such a complex, and $\GFlat(R)$ denotes the class of Gorenstein flat modules. The \emph{Gorenstein flat dimension} $\Gfd_RM$ is the shortest length of a resolution of $M$ whose terms lie in $\GFlat(R)$, with $\Gfd_RM=\infty$ if $M$ admits no finite such resolution.

Over a general ring the class $\GFlat(R)$ is considerably better behaved than one might expect. Indeed, by the singular compactness theorem of \v{S}aroch--\v{S}\v{t}ov\'i\v{c}ek \cite[Corollary~4.12]{SS}, $\GFlat(R)$ is closed under extensions over \emph{every} ring; in other words, every ring is GF-closed. Consequently, $\GFlat(R)$ is projectively resolving and closed under direct summands \cite{SS,Ben}, so that the Gorenstein flat dimension may be computed by dimension shifting along flat resolutions: if $\Gfd_RM\le n$, then the $n$-th syzygy of any flat resolution of $M$ is Gorenstein flat and $\Tor^R_i(I,M)=0$ for every $i>n$ and every injective right $R$-module $I$ (see \cite[Theorem~2.8]{Ben} and \cite[Lemma~2.4]{Ben} for $n=0$). A further consequence is that the following criterion of Bennis, originally established under a GF-closedness hypothesis, is now available without any restriction on $R$.

\begin{proposition}[{\cite[Theorem~2.3]{Ben}}, {\cite[Corollary~4.12]{SS}}]\label{prop2.1} For every short exact sequence of left $R$-modules $0\rightarrow G_1\rightarrow G_0 \rightarrow M \rightarrow 0,$ where $G_0$ and $G_1$ are Gorenstein flat, if $\Tor_1^R(I,M)=0$ for every injective right $R$-module $I$, then $M$ is Gorenstein flat.
\end{proposition}

\subsection{PGF modules}

The two definitions above may be combined: one may retain the projectivity of the terms of the complex, as in the Gorenstein projective case, while imposing the tensor exactness condition, as in the Gorenstein flat case. The resulting class of modules was singled out by \v{S}aroch--\v{S}\v{t}ov\'i\v{c}ek \cite{SS}.

\begin{definition}\label{def2.2} An $R$-module $M$ is \emph{projectively coresolved Gorenstein flat}, abbreviated to \emph{PGF}, if there is an acyclic complex $\mathbf P$ of projective $R$-modules having $M$ among its cycles, such that $I\otimes_R\mathbf P$ is acyclic for every injective right $R$-module $I$ (see \cite{SS}). The class of PGF $R$-modules is denoted by $\PGF(R)$.
\end{definition}

Since projective modules are flat, every PGF module is Gorenstein flat; and every PGF module is Gorenstein projective by \cite[Theorem~4.4]{SS}. Thus
\[
\PGF(R)\;\subseteq\;\GProj(R)\cap\GFlat(R),
\]
and $\PGF(R)$ contains all projective modules. It is not known whether the above inclusion is an equality, nor whether $\PGF(R)=\GProj(R)$ in general; as observed in \cite[Section~4]{SS}, the equality $\PGF(R)=\GProj(R)$ holds if and only if every Gorenstein projective module is Gorenstein flat, which is a long-standing open problem.

The class $\PGF(R)$ can serve as a substitute for $\GProj(R)$ over an arbitrary ring thanks to the following closure properties, all of which are established in \cite{SS}. The class $\PGF(R)$ is projectively resolving and it is closed under arbitrary direct sums and under direct summands. Moreover, if $\PGF(R)^{\perp}$ denotes the class of $R$-modules $L$ such that $\Ext^1_R(X,L)=0$ for every $X\in\PGF(R)$, then the pair $(\PGF(R),\PGF(R)^{\perp})$ is a complete hereditary cotorsion pair in the category of $R$-modules (see \cite[Theorem~4.9]{SS}). No analogous result is known for $\GProj(R)$ over a general ring.
The cotorsion pair above plays a key role in establishing the properties of PGF-dimension used below.

Finally, we note the following observation of \cite[Lemma~2.1]{KS}, which is responsible for part (v) of Proposition~\ref{prop2.3} below: \emph{every PGF module of finite flat dimension is projective}.

\subsection{The PGF-dimension}

The \emph{PGF-dimension} $\PGFdim_RM$ of an $R$-module $M$ is the least length of a resolution of $M$ by PGF modules, that is, the smallest $n\ge0$ for which there is an exact sequence
\[
0\rightarrow X_n\rightarrow X_{n-1}\rightarrow\cdots\rightarrow X_0\rightarrow M\rightarrow0
\]
with $X_0,\dots,X_n\in \PGF(R)$; as usual, $\PGFdim_RM=\infty$ if no finite resolution of this shape exists. Thus $\PGFdim_RM=0$ if and only if $M$ is PGF.

The invariant so defined has the expected properties. The properties collected in the next proposition are the verbatim analogues, with $\PGF(R)$ in the role of $\GProj(R)$, of the familiar properties of the Gorenstein projective dimension \cite{Hol}, and they are established in the same manner. We record here only the statements that will be used in the sequel (see \cite[Propositions~2--5, Theorem~10 and Corollary~13]{DE}); part (v), in the form stated below, rests on \cite[Lemma~2.1]{KS}.

\begin{proposition}\label{prop2.3}
Let $R$ be a ring, let $M$ be an $R$-module and let $n\ge0$ be an integer.
\begin{enumerate}
\item[(i)]We have
$\Gpd_RM\le\PGFdim_RM\le\pd_RM$ and $\Gfd_RM\le\PGFdim_RM$.
\item[(ii)] The following assertions are equivalent:
\begin{enumerate}
\item[(a)] $\PGFdim_RM\le n$;
\item[(b)] there is an exact sequence $0\to K\to X_{n-1}\to\cdots\to X_0\to M\to0$ with all $X_j\in\PGF(R)$ and $K\in\PGF(R)$;
\item[(c)] for \emph{every} exact sequence $0\to K\to X_{n-1}\to\cdots\to X_0\to M\to0$ with all $X_j\in\PGF(R)$, the module $K$ is PGF;
\item[(d)] the $n$-th syzygy of any projective resolution of $M$ is PGF.
\end{enumerate}
\item[(iii)]
$\PGFdim_R\Bigl(\bigoplus_{\lambda\in\Lambda}M_\lambda\Bigr)=\sup_{\lambda\in\Lambda}\PGFdim_RM_\lambda$, for any family $(M_\lambda)_{\lambda\in\Lambda}$ of $R$-modules.
\item[(iv)]If $0\to M'\to M\to M''\to0$ is a short exact sequence of $R$-modules, then
\begin{enumerate}
\item[(a)]$\PGFdim_RM\le\max\{\PGFdim_RM',\PGFdim_RM''\}$,
\item[(b)]$\PGFdim_RM'\le\max\{\PGFdim_RM,\PGFdim_RM''-1\}$,
\item[(c)]$\PGFdim_RM''\le\max\{\PGFdim_RM,\PGFdim_RM'+1\}$.
\end{enumerate}
In particular, if two of the three dimensions are finite, then so is the third.
\item[(v)]If $\fd_RM<\infty$, then $\PGFdim_RM=\pd_RM$. In particular, the PGF-dimension and the projective dimension coincide on modules of finite projective dimension.
\item[(vi)] $\PGFdim_RM\le n$ if and only if there exists a short exact sequence of $R$-modules $0\rightarrow M\rightarrow D\rightarrow G\rightarrow0$, in which $G$ is PGF and $\pd_RD\le n$.
\end{enumerate}
\end{proposition}

We shall also use the following sufficient condition for bounded
PGF-dimension, formulated in terms of \emph{strongly $n$-PGF} modules: if a module $N$ fits into a short exact sequence of $R$-modules $0\to N\to X\to N\to0$ with $\pd_RX\le p$ and satisfies $\Tor^R_{p+1}(I,N)=0$ for every injective right $R$-module $I$, then $\PGFdim_RN\le p$ (see Appendix~\ref{sec:appendix}).

\subsection{The invariants \texorpdfstring{$\sfli$}{sfli} and \texorpdfstring{$\Gwgl$}{Gwgl.dim}}

Let $R$ be a ring. Following Gedrich--Gruenberg \cite{GG}, one sets $\spli R=\sup_I\{\pd_RI\}$ and $\silp R=\sup_P\{\id_RP\}$, where $I$ runs over the injective (left) $R$-modules and $P$ over the projective ones; similarly, one sets $\sfli R=\sup_I\{\fd_RI\}$, with $I$ again running over the injective $R$-modules (see \cite{Emm,ET2}). The corresponding invariants of the opposite ring, such as $\sfli R^{\mathrm{op}}$, are defined using right $R$-modules. One defines the \emph{Gorenstein weak global dimension} of $R$ by $\Gwgl R=\sup\{\Gfd_RM:M\ \text{an } R\text{-module}\}$. Similarly, the \emph{Gorenstein global dimension} of $R$ is defined as $\Ggl R$ $=\sup\{\Gpd_RM:M\ \text{an } R\text{-module}\}$, which also equals the supremum of the Gorenstein injective dimensions of all $R$-modules (see \cite[Theorem~1.1]{BM2}), and which is finite if and only if $\spli R$ and $\silp R$ are finite, in which case $\Ggl R=\spli R=\silp R$ (see \cite[Theorem~4.1]{Emm}). By \cite[Theorem~5.3]{Emm}, $\Gwgl R<\infty$ if and only if both $\sfli R<\infty$ and $\sfli R^{\mathrm{op}}<\infty$, in which case $\Gwgl R=\sfli R=\sfli R^{\mathrm{op}}$; moreover, by \cite[Corollary~1.5]{CET} the invariant $\Gwgl$ is left--right symmetric in general. For a ring $R$ such that $R\cong R^{\mathrm{op}}$ (e.g.\ a group algebra $kG$, via $g\mapsto g^{-1}$) one has $\sfli R=\sfli R^{\mathrm{op}}$, and hence $\Gwgl R=\sfli R$, the two sides being finite or infinite simultaneously. Finally, for a commutative ring $k$ one has $\sfli k\le\wgl k\le\gldim k$, so that the hypothesis $\sfli k<\infty$ occurring in Corollary~\ref{corC} is a mild one; it is satisfied, for instance, by $\Z$, by every field, by every ring of finite weak global dimension and by every quasi-Frobenius ring.

\subsection{Two cardinality lemmas}\label{subsec:cardinal}

For an infinite cardinal $\kappa$, a module is \emph{$\kappa$-generated} if it admits a generating set of cardinality at most $\kappa$, and \emph{$\kappa$-presented} if it admits a presentation with at most $\kappa$ generators \emph{and} at most $\kappa$ relations. The following two lemmas are the cardinal counterparts of familiar facts about finitely presented modules.

\begin{lemma}\label{lem2.4}
Let $m$ be a nonnegative integer and $T$ be an $\alephm$-presented flat $R$-module. Then, $\pd_RT\le m+1$.
\end{lemma}

\begin{proof}
Since $T$ is $\aleph_m$-presented, it admits a presentation with at most
$\aleph_m$ generators and at most $\aleph_m$ relations, i.e.\ with fewer than $\aleph_{m+1}$ generators and relations. By Jensen \cite[Proposition~5.3]{Jen}, every flat module with fewer than $\aleph_{m+1}$ generators and relations has projective dimension at
most $m+1$; see also \cite[Theorem~2.45]{Oso}, where it is shown more generally that a flat module which is the quotient of a free module by an $\alephm$-generated submodule has projective dimension at most $m+1$. Hence $\pd_RT\le m+1$, as needed.
\end{proof}

The following lemma is the cardinal analogue of the standard fact that the quotient of a finitely presented module by a finitely generated submodule is finitely presented.

\begin{lemma}\label{lem2.5} Let $\kappa$ be an infinite cardinal. If $T$ is a $\kappa$-presented
$R$-module and $N\leq T$ is a $\kappa$-generated submodule, then
$T/N$ is $\kappa$-presented.
\end{lemma}

\begin{proof}
Since $T$ is $\kappa$-presented, there exists an exact sequence of the form
\[
R^{(J)} \rightarrow R^{(I)}
\xrightarrow{\pi} T \rightarrow 0,
\qquad |I|,|J|\leq \kappa.
\]
Let
$
N=\langle n_\alpha\mid \alpha\in A\rangle$, $|A|\leq \kappa$.
We choose for every $\alpha\in A$ an element $x_\alpha\in R^{(I)}$
such that $\pi(x_\alpha)=n_\alpha$. Then the composite
$
R^{(I)}\xrightarrow{\pi}T\rightarrow T/N
$
is surjective, and its kernel is generated by
$
\im\bigl(R^{(J)}\to R^{(I)}\bigr)$, together with the elements $x_\alpha,\ \alpha\in A$.
Therefore, this kernel is generated by at most
$
|J|+|A|\leq \kappa+\kappa=\kappa
$
elements, since $\kappa$ is infinite. Consequently, $T/N$ admits a presentation
with at most $\kappa$ generators and at most $\kappa$ relations, and hence
$T/N$ is $\kappa$-presented.
\end{proof}

\section{The approximation step}\label{sec:step}

The next statement is the small-cardinality analogue of \cite[Proposition~3.1]{St}, with ``finitely presented'' relaxed to ``$\alephm$-presented'' and ``finitely generated free'' relaxed to ``$\alephm$-presented flat''.
\begin{proposition}\label{prop3.1}
Let $m$ be a nonnegative integer, $R$ be an arbitrary ring and $M$ be an $\alephm$-presented Gorenstein flat $R$-module. Then:
\begin{enumerate}
 \item[(i)] There exists a short exact sequence of $R$-modules
\[
0\rightarrow M\rightarrow T\rightarrow C\rightarrow0,
\]
where $T$ is an $\alephm$-presented flat module (so $\pd_RT\le m+1$), $C$ is an $\alephm$-presented Gorenstein flat module, and the sequence remains exact after applying $I\otimes_R-$ for every injective right $R$-module $I$.

\item[(ii)]There exists an exact sequence of $R$-modules
\[
0\to M\to T_{-1}\to T_{-2}\to\cdots,
\]
where each $T_i$ is an $\alephm$-presented flat module (so $\pd_RT_i\le m+1$) and the images $C_i=\im(T_i\to T_{i-1})$ are $\alephm$-presented Gorenstein flat modules.
\item[(iii)] Every projective resolution $\mathbf P$ of $M$
\[\mathbf P:\ \cdots \rightarrow P_{j} \rightarrow P_{j-1}\rightarrow \cdots \rightarrow P_0 \rightarrow M \rightarrow 0\]
can be completed to an $F$-totally acyclic complex $\mathbf F$ of flat modules
\[\mathbf F:\ \cdots\rightarrow P_{j} \rightarrow P_{j-1}\rightarrow \cdots \rightarrow P_0\to T_{-1}\to T_{-2}\to\cdots\] with $M=\im(P_0\to T_{-1})$, where $T_i$ is an $\alephm$-presented flat $R$-module (so $\pd_RT_i\le m+1$) for every $i\le-1$.
\end{enumerate}
\end{proposition}


\begin{proof}

(i) Let $\kappa=\aleph_m$. Since $M$ is Gorenstein flat, there exists a
short exact sequence of $R$-modules $0\rightarrow M \xrightarrow{\,u\,} F_0 \rightarrow M'
\rightarrow 0,$
where $F_0$ is flat and $M'$ is Gorenstein flat. Moreover,
$
\Tor_1^R(I,M')=0
$
for every injective right $R$-module $I$ and hence the map
$
1_I\otimes u:I\otimes_R M\rightarrow I\otimes_R F_0$
is injective for every such $I$. By the Govorov--Lazard theorem, we may write $F_0=\varinjlim_{j\in J}L_j,$
where $(J,\leq)$ is directed and every $R$-module $L_j$ is finitely generated and
free. Let $\omega_j:L_j\rightarrow F_0$
be the canonical maps and
$\varphi_{jj'}:L_j\rightarrow L_{j'}$,
$j\leq j'$, the transition homomorphisms.
Since $M$ is $\kappa$-presented, we may fix a presentation $R^{(U)}\rightarrow R^{(S)}
\xrightarrow{\pi}M\rightarrow 0$
with
$|S|,|U|\leq\kappa.$
Writing $x_s=\pi(e_s)$, we may identify the $R$-module
$
M$ with the quotient $R^{(S)}/N,$
where $N$ is generated by the relations $r_t=\sum_{s\in S_t}a_{t,s}e_s$, $t\in U$, with each $S_t$ a finite subset of $S$. For every $s\in S$, we choose $j_s\in J$ and $y_s\in L_{j_s}$ such that
$\omega_{j_s}(y_s)=u(x_s).$
We shall construct a directed subset $J_0\subseteq J$ of cardinality
at most $\kappa$ such that the elements $y_s$ and all the defining
relations of $M$ are already realized in the restricted system
indexed by $J_0$.
We start with $J^{(0)}=\{j_s\mid s\in S\}.$
Inductively, we enlarge the set constructed so far by performing the
following two operations. Firstly, for every pair of indices already
chosen, we adjoin an upper bound in $J$. Secondly, for every relation
$r_t=\sum_{s\in S_t}a_{t,s}e_s,$ we
choose an index $j$ dominating all $j_s$, $s\in S_t$, and consider the element
$z_t=
\sum_{s\in S_t}
a_{t,s}\varphi_{j_sj}(y_s)\in L_j.$
Since
$
\omega_j(z_t)
=
u\left(\sum_{s\in S_t}a_{t,s}x_s\right)
=0,$
the defining property of a directed colimit yields some $j'\geq j$
such that
$
\varphi_{jj'}(z_t)=0.$ We adjoin both $j$ and $j'$. At each stage at most $\kappa$ indices are added, since
$\kappa^2=\kappa$. Repeating this construction over $\omega$ steps and
putting $J_0=\bigcup_{n<\omega}J^{(n)},$
we obtain a directed subset $J_0\subseteq J$ satisfying
$
|J_0|\leq\kappa.$
Let
$
T=\varinjlim_{j\in J_0}L_j$ and denote by
$
\lambda_j:L_j\rightarrow T
$
the structural maps of this direct colimit. The inclusion $J_0\subseteq J$ induces a
canonical homomorphism $\iota:T\rightarrow F_0$
such that $\iota \circ \lambda_j=\omega_j.$ For $s\in S$, we set
$
\widetilde y_s=\lambda_{j_s}(y_s)\in T.$
By construction, all the defining relations of $M$ vanish on the
family $(\widetilde y_s)_{s\in S}$. Hence the assignment
$x_s\mapsto \widetilde y_s$
induces a well-defined homomorphism $g:M\rightarrow T.$
Moreover, $\iota \circ g=u.$
Since $u$ is injective, it follows that $g$ is injective as well. Thus, setting
$C=\coker(g),$
we obtain a short exact sequence of $R$-modules
\begin{equation}\label{eq1}
	0\rightarrow M\xrightarrow{\,g\,}T\rightarrow C
\rightarrow 0.
\end{equation}
Since $T$ is a directed colimit of finitely generated free modules, it is flat. We will next show that $T$ is $\kappa$-presented. The standard presentation of a directed colimit gives an exact sequence of $R$-modules
$$\bigoplus_{\substack{j,j'\in J_0\\ j\leq j'}}L_j
\rightarrow
\bigoplus_{j\in J_0}L_j
\rightarrow T\rightarrow 0.$$
As $|J_0|\leq\kappa$, there are at most $\kappa$ comparable pairs
$(j,j')$, and each $L_j$ is finitely generated and free. Hence both free modules in the above presentation have rank at most $\kappa$, and
therefore $T$ is $\kappa$-presented. Using Lemma~\ref{lem2.4}, we deduce that
$\pd_R T\leq m+1.$ Furthermore, $C$ is $\kappa$-presented. Indeed, Lemma~\ref{lem2.5} applies since $T$ is
$\kappa$-presented and $g(M)$ is generated by at most $\kappa$
elements.
It remains to show that the exact sequence \eqref{eq1} remains exact after applying the functor $I\otimes_R -$, for every injective right $R$-module $I$. Let $I$ be an
injective right $R$-module. Since
$u=\iota \circ g,$
we have
$
1_I\otimes u
=
(1_I\otimes\iota)\circ(1_I\otimes g).$
Moreover $1_I\otimes u$ is injective and hence $1_I\otimes g$ is injective as well. We consider now the long exact Tor sequence \begin{equation*} \cdots \rightarrow \Tor_1^R(I,T)\rightarrow \Tor_1^R(I,C)\rightarrow I\otimes_RM\xrightarrow{1_I \otimes g}I\otimes_RT\rightarrow I\otimes_RC\rightarrow 0.\end{equation*} Since $T$ is flat, we get the exact sequence
\begin{equation*}
	0\rightarrow \Tor_1^R(I,C)\rightarrow I\otimes_RM\xrightarrow{1_I \otimes g}I\otimes_RT\rightarrow I\otimes_RC\rightarrow 0.\end{equation*} The injectivity of $1_I \otimes g$ yields $\Tor_1^R(I,C)=0$ and hence the sequence $0\rightarrow I\otimes_RM\rightarrow I\otimes_RT\rightarrow I\otimes_RC\rightarrow 0$ is exact
for every injective right $R$-module $I$.  Moreover, $T$ is flat and $M$ is Gorenstein flat and hence the short exact sequence \eqref{eq1} and Proposition~\ref{prop2.1} yield that the $R$-module $C$ is Gorenstein flat, as needed.

(ii) This follows inductively from (i), splicing the short exact sequences obtained by a repeated application of (i). We note that the module $C$ of each step is an $\alephm$-presented Gorenstein flat module, so the step can be repeated.

(iii) We splice the projective resolution $\mathbf P$ of $M$ with the exact sequence of (ii). The resulting complex is an acyclic complex of flat modules \[\mathbf F:\ \cdots\rightarrow P_{j} \rightarrow P_{j-1}\rightarrow \cdots \rightarrow P_0\to T_{-1}\to T_{-2}\to\cdots\] with $M=\im(P_0\to T_{-1})$, such that $T_i$ is an $\alephm$-presented flat $R$-module for every $i\le-1$. We denote by $F_i$ the terms of $\mathbf F$ (so $F_i=P_i$ for $i\ge0$ and $F_i=T_i$ for $i\le-1$) and by $Z_i=\ker(F_i\to F_{i-1})$, $i\in\mathbb{Z}$, its syzygies, so that there are short exact sequences $0\to Z_i\to F_i\to Z_{i-1}\to0$ for all $i\in\mathbb{Z}$. For $j\ge1$, let $K_j$ be the $j$-th syzygy of $M$ in $\mathbf P$, and set $K_0=M$. Then $Z_{i-1}=K_i$ for every $i\ge0$ (in particular, $Z_{-1}=M$), whereas $Z_{i-1}=C_i=\im(T_i\to T_{i-1})$ for every $i\le-1$. Since $M$ is Gorenstein flat, the $R$-module $K_j$ is Gorenstein flat for every $j\geq 0$, and hence $\Tor_1^R(I,K_j)=0$ for every injective right $R$-module $I$ (see Section~\ref{sec:prelim}). The modules $C_i$ are Gorenstein flat by (ii), and hence $\Tor_1^R(I,C_i)=0$ as well. Therefore, $\Tor_1^R(I,Z_{i-1})=0$ for every $i\in\mathbb{Z}$, and so every short exact sequence $0\to Z_i\to F_i\to Z_{i-1}\to0$ remains exact after the application of the functor $I\otimes_R-$. It follows that $\mathbf F$ is $F$-totally acyclic, as needed.
\end{proof}

\section{\texorpdfstring{$F$}{F}-totally acyclic complexes and the proof of Theorem~\texorpdfstring{\ref{theoA}}{A}}\label{sec:cover}

In this section we prove Theorem~\ref{theoA}, we record the form in which it will be applied to groups, and we deduce from it a Gorenstein counterpart of the Osofsky--Jensen inequality for the global dimension of a ring.

\subsection{\texorpdfstring{The proof of Theorem~\ref{theoA}}{The proof of Theorem A}}

The following lemma is the main technical tool of the paper.
It shows that a bound on the projective dimensions of the terms
of an $F$-totally acyclic complex of flat modules yields a bound
on the PGF-dimensions of its syzygies. The idea is to sum all the syzygies into a single module, which is then periodic modulo a module of finite projective dimension, and to apply the results of Appendix~\ref{sec:appendix}.

\begin{lemma}\label{lem4.1}
Let $R$ be a ring, $p$ be a nonnegative integer and $\mathbf F$ be an $F$-totally acyclic complex of flat $R$-modules $F_i$ such that $\pd_RF_i\le p$ for every $i\in \mathbb{Z}$. Then, every syzygy module $Z_i$ of the complex $\mathbf F$ satisfies $\PGFdim_RZ_i\le p$, $i\in\mathbb{Z}$. Moreover, for every $i\in\mathbb{Z}$, there exists an exact sequence
\[0\rightarrow X\rightarrow Q_{p-1}\rightarrow\cdots\rightarrow Q_0\rightarrow Z_i\rightarrow0
\]
in which $X$ is PGF and each $Q_j$ is projective.
\end{lemma}

\begin{proof}We consider the short exact sequences
\begin{equation}\label{eq2}0\to Z_i\to F_i\to Z_{i-1}\to0,\end{equation} for every $i\in\mathbb{Z}$. Summing the exact sequences \eqref{eq2} over $i\in\Z$, we deduce the exact sequence of $R$-modules
\begin{equation}\label{eq3}0\rightarrow \bigoplus_{i\in \mathbb{Z}} Z_i\rightarrow \bigoplus_{i\in \mathbb{Z}}F_i\rightarrow  \bigoplus_{i\in \mathbb{Z}} Z_{i-1}\rightarrow 0.\end{equation}
Using the reindexing $\bigoplus_iZ_{i-1}=\bigoplus_iZ_i=:N$ the short exact sequence \eqref{eq3} takes the form $0\rightarrow N\rightarrow Y\rightarrow N\rightarrow0,$ where
$Y:=\bigoplus_{i\in \mathbb{Z}}F_i$ and $\pd_RY=\sup_i\pd_RF_i\le p$. Every syzygy module $Z_i$ is, by its definition, a cycle of the $F$-totally acyclic complex $\mathbf F$, and hence a Gorenstein flat $R$-module. Consequently, $\Tor^R_j(I,Z_i)=0$ for every $i\in\mathbb Z$, every $j>0$ and every injective right $R$-module $I$ (see \cite[Lemma~2.4]{Ben}, which is valid over an
arbitrary ring). It follows that $\Tor_j^R(I,N)\cong \bigoplus_{i\in\mathbb{Z}}\Tor_j^R(I,Z_{i})=0$ for every  $j>0$ and every injective right $R$-module $I$. Therefore, the $R$-module $N$ is \emph{strongly $p$-PGF} in the sense of Definition~\ref{defA.1}, and hence $\PGFdim_RN\le p$ by Proposition~\ref{propA.4}(i). We conclude that $\PGFdim_RZ_i\le p$ for every $i$, as the PGF-dimension of a direct sum is the supremum of the PGF-dimensions of its direct summands (see Proposition~\ref{prop2.3}(iii)). The final assertion of the lemma then follows from Proposition~\ref{prop2.3}(ii), applied to a partial projective resolution of $Z_i$ of length $p$.
\end{proof}


\begin{proof}[Proof of Theorem~\ref{theoA}]
 Let $M$ be an $\alephm$-presented Gorenstein flat module over a ring $R$. Invoking Proposition~\ref{prop3.1}(iii), we obtain an $F$-totally acyclic complex of flat $R$-modules
\[
\mathbf F:\ \cdots\rightarrow P_1\rightarrow P_0\rightarrow T_{-1}\rightarrow T_{-2}\rightarrow\cdots,
\qquad M=\im(P_0\to T_{-1}),
\] where $P_i$ is projective for every $i\geq 0$ and $\pd_RT_j\leq m+1$ for every $j\le-1$. Then, Lemma~\ref{lem4.1} with $p=m+1$ yields $\PGFdim_RM\le m+1$, as needed.
\end{proof}


\begin{remark}\label{rem4.2}The bound in Theorem~\ref{theoA} is sharp for every nonnegative
integer $m$, even for flat modules. Indeed, the construction
in \cite[proof of Proposition~2.62]{Oso}, with
$n=m+2$, provides a valuation domain $R$ and an ideal
$I=\bigcup_{\alpha<\omega_m}I_\alpha$,
where $\omega_m$ is the initial ordinal of cardinality
$\aleph_m$ and $(I_\alpha)_{\alpha<\omega_m}$ is a strictly
increasing chain of nonzero principal ideals.
By \cite[Theorem~2.57]{Oso}, this ideal satisfies
$\pd_RI=m+1$.
Each $I_\alpha$ is free of rank one, so $I$ is flat.
Moreover, the standard presentation of this direct limit
uses at most $\aleph_m$ generators and relations.
Thus $I$ is an $\aleph_m$-presented flat module, and hence
an $\aleph_m$-presented Gorenstein flat module.
Proposition~\ref{prop2.3}(v) now gives
$\PGFdim_RI
=\pd_RI
=m+1$,
showing that the bound is attained.
For $m=0$, one may simply take $\mathbb Q$, which is a
countably presented flat $\mathbb Z$-module satisfying
$\PGFdim_\Z\Q
=\Gpd_\Z\Q
=\pd_\Z\Q
=1$.
\end{remark}

Theorem~\ref{theoA} concerns modules of small presentation. By dimension shifting it propagates to modules of finite Gorenstein flat dimension admitting a projective resolution with small terms; this is the form in which it will be applied to groups.

\begin{corollary}\label{cor4.3}
Let $R$ be a ring, $m,d\ge0$ be nonnegative integers and $M$ be an $R$-module such that $\Gfd_RM\le d$. Assume that there exists an exact sequence of $R$-modules
\[
P_{d+1}\rightarrow P_d\rightarrow P_{d-1}\rightarrow\cdots\rightarrow P_0\rightarrow M\rightarrow0
\]
with all $P_j$ projective and $P_d$, $P_{d+1}$ $\alephm$-generated. Then, $\PGFdim_RM\le d+m+1$. In particular, $\Gpd_RM\le d+m+1$.
\end{corollary}

\begin{proof}Let $\delta_{d+1}:P_{d+1}\rightarrow P_d$ denote the differential of the exact sequence and let $K_d=\coker\delta_{d+1}$.
We claim that $K_d$ is $\aleph_m$-presented. Since $P_d$ is $\aleph_m$-generated, there exists an epimorphism $\pi:F\twoheadrightarrow P_d$, where $F$ is a free $R$-module of rank at most $\aleph_m$.
As $P_d$ is projective, the epimorphism $\pi$ splits. Hence $F\cong \ker\pi\oplus P_d$. Therefore, $\ker\pi$ is a direct summand of the
$\aleph_m$-generated module $F$, and is consequently itself $\aleph_m$-generated. We consider now the canonical quotient map $q:P_d\twoheadrightarrow K_d$. Then, the map $q\circ\pi:F\twoheadrightarrow K_d$ is an epimorphism. To show that the $R$-module $K_d$ is $\aleph_m$-presented, it is enough to prove that $\ker(q\circ\pi)$ is $\aleph_m$-generated. Since $\ker q=\im\delta_{d+1}$,
we have $\ker(q\circ\pi)=\pi^{-1}\!\bigl(\im\delta_{d+1}\bigr)$. Moreover, the $R$-module $P_{d+1}$ is $\aleph_m$-generated and hence its image in $P_d$
is also $\aleph_m$-generated. Thus, we can choose a generating family $\{y_\alpha\}_{\alpha\in A}
\subseteq \im\delta_{d+1}$, $|A|\leq\aleph_m$,
and, for each $\alpha\in A$, we choose a lift $\widetilde y_\alpha\in F$ such that $\pi(\widetilde y_\alpha)=y_\alpha.$ Then,
\begin{equation}\label{eq4}
	\ker(q\circ\pi)=\ker\pi+\sum_{\alpha\in A}R\widetilde y_\alpha.\end{equation}
Indeed, the inclusion from the right side to the left in \eqref{eq4} is immediate. Conversely, if $x\in\ker(q\circ\pi)$, then $\pi(x)\in\im\delta_{d+1}$, so there exist $r_1,\dots,r_t\in R$ and $\alpha_1,\dots,\alpha_t\in A$ such that $\pi(x)=\sum_{i=1}^t r_i y_{\alpha_i}$. Hence
$\pi\left(x-\sum_{i=1}^t r_i\widetilde y_{\alpha_i}\right)=0$,
and therefore
$x-\sum_{i=1}^t r_i\widetilde y_{\alpha_i}\in\ker\pi$.
This proves the description of $\ker(q\circ\pi)$ in \eqref{eq4}.
Now $\ker\pi$ is $\aleph_m$-generated, and the family
$\{\widetilde y_\alpha\}_{\alpha\in A}$ has cardinality at most
$\aleph_m$. Since $\aleph_m$ is infinite,
it follows that $\ker(q\circ\pi)$ is $\aleph_m$-generated, as needed.
 Moreover, the exact sequence $0\to K_d\to P_{d-1}\to\cdots\to P_0\to M\to0$ exhibits $K_d$ as a $d$-th syzygy of $M$ in a projective resolution (for $d=0$ we set $K_0=M$). Since $\Gfd_RM\le d$ and every ring is GF-closed (see \cite[Corollary~4.12]{SS}), $K_d$ is Gorenstein flat by \cite[Theorem~2.8]{Ben}. Theorem~\ref{theoA} therefore gives $\PGFdim_RK_d\le m+1$. Finally, applying \cite[Lemma~1.2]{St0}, we deduce that $\PGFdim_RM=\PGFdim_RK_0\le\PGFdim_RK_d+d\le d+m+1$. The last assertion follows from Proposition~\ref{prop2.3}(i).
\end{proof}

\begin{corollary}\label{cor4.4}
Let $m$ be a nonnegative integer and $R$ be a ring all of whose left ideals are $\alephm$-generated; for instance, any ring of cardinality at most $\alephm$. Then, every $\alephm$-generated $R$-module $M$ satisfies
\[
\Gpd_RM\;\le\;\PGFdim_RM\;\le\;\Gfd_RM+m+1 .
\]
In particular, over a countable ring every countably generated Gorenstein flat module has PGF-dimension, and hence Gorenstein projective dimension, at most $1$.
\end{corollary}

\begin{proof} We may assume that $\Gfd_R M=d<\infty$. By hypothesis, every left ideal of $R$ is $\aleph_m$-generated. Hence every submodule of a free $R$-module of rank at most
$\aleph_m$ is again $\aleph_m$-generated (see \cite[Lemma~2.46]{Oso}). 
Since $M$ is $\aleph_m$-generated, we choose an epimorphism $P_0=R^{(X_0)}\twoheadrightarrow M$ with $|X_0|\leq\aleph_m$. Let $K_1=\ker(P_0\to M)$. As $K_1$ is a submodule of the free module $P_0$, whose rank is at most $\aleph_m$, it follows that $K_1$ is $\aleph_m$-generated.
Hence there exists an epimorphism
$P_1=R^{(X_1)}\twoheadrightarrow K_1$
with $|X_1|\leq\aleph_m$.
Composing this map with the inclusion $K_1\hookrightarrow P_0$, we
obtain an exact sequence $P_1\rightarrow P_0\rightarrow M\rightarrow 0$. Proceeding inductively, we suppose that
$P_i=R^{(X_i)}$ has been constructed with $|X_i|\leq\aleph_m$.
Then, the next syzygy $K_{i+1}=\ker(P_i\to P_{i-1})$
is a submodule of the free module $P_i$, and hence is
$\aleph_m$-generated. Therefore, we may choose an epimorphism
$P_{i+1}=R^{(X_{i+1})}\twoheadrightarrow K_{i+1}$ with $|X_{i+1}|\leq\aleph_m$.
Iterating this construction yields a free resolution
$$\cdots\rightarrow P_2\rightarrow P_1
\rightarrow P_0\rightarrow M\rightarrow 0,
$$
such that $P_i\cong R^{(X_i)}$ and $|X_i|\leq\aleph_m$ for every $i\geq 0$. In particular, $P_d$ and $P_{d+1}$ are
$\aleph_m$-generated. Hence Corollary~\ref{cor4.3} applies and yields $\PGFdim_R M\leq d+m+1$.
Since $\Gfd_R M=d$,
we conclude that $\Gpd_R M \leq\PGFdim_R M\leq\Gfd_R M+m+1$, as needed. If $R$ is countable, then we may take $m=0$. Thus, for every
countably generated Gorenstein flat $R$-module $M$, $\Gfd_R M=0$, and therefore $\PGFdim_R M\leq 1$. Consequently, $\Gpd_R M\leq 1$ as well.
\end{proof}

\subsection{A Gorenstein version of Osofsky's inequality}

Corollary~\ref{cor4.4} bounds the Gorenstein projective dimension of an $\alephm$-generated module by its Gorenstein flat dimension, up to the additive constant $m+1$. Passing to the supremum over the cyclic modules $R/I$ turns this pointwise estimate into an inequality between two \emph{global} invariants, and yields a Gorenstein counterpart of a classical theorem of Osofsky and Jensen. Recall that, for a ring all of whose left ideals are $\alephm$-generated, one has $\gldim R\le\wgl R+m+1$, by \cite[Corollary~2.47]{Oso} (see also \cite[Corollary~1.4]{Osofsky1968}), the case $m=0$ of rings with countably generated ideals being due to Jensen \cite{Jen2}.

We recall from Section~\ref{sec:prelim} the definitions of the Gorenstein global dimension and the Gorenstein weak global dimension of a ring $R$ as
$\Ggl R=\sup\{\Gpd_RM:M\ \text{an } R\text{-module}\}$ and $\Gwgl R=\sup\{\Gfd_RM:M\ \text{an } R\text{-module}\}$, respectively. The \emph{global PGF-dimension} $\mathrm{PGF\text{-}gl.dim}\,R=\sup\{\PGFdim_R M:M\ \text{an } R\text{-module}\}$ was introduced and studied in \cite{DE,ElM}. Although the inequality $\Gpd_RM\le\PGFdim_RM$ of Proposition~\ref{prop2.3}(i) is not known to be an equality in general---it is one whenever the right-hand side is finite, by \cite[Corollary~13(ii)]{DE}---the two \emph{global} invariants coincide: $\Ggl R=\mathrm{PGF\text{-}gl.dim}\,R$ for every ring $R$, by \cite[Theorem~4.16]{ElM}.

\begin{corollary}[Osofsky's inequality for the Gorenstein dimensions]\label{cor4.5}
Let $m\ge0$ and let $R$ be a ring all of whose left ideals are $\alephm$-generated; for instance, any ring of cardinality at most $\alephm$. Then
\[
\Ggl R\;=\;\mathrm{PGF\text{-}gl.dim}\,R\;\le\;\Gwgl R+m+1 .
\]
In particular, if $\Gwgl R<\infty$, then $\Ggl R<\infty$. For $m=0$, that is, when every left ideal of $R$ is countably generated, one has $\Ggl R\le\Gwgl R+1$.
\end{corollary}

\begin{proof}
We may assume that $\Gwgl R<\infty$, as otherwise there is nothing to prove. By the PGF analogue of Auslander's theorem \cite[Corollary~4.19]{ElM}, valid over any ring,
\begin{equation}\label{eq5}
	\mathrm{PGF\text{-}gl.dim}\,R=\sup\{\PGFdim_R(R/I):I\text{ is a left ideal of } R\}.
\end{equation}
Now each $R/I$ is cyclic and hence $\alephm$-generated. Therefore, Corollary~\ref{cor4.4} yields
\[
\PGFdim_R(R/I)\;\le\;\Gfd_R(R/I)+m+1\;\le\;\Gwgl R+m+1
\]
for every left ideal $I$. Using the characterization \eqref{eq5} of $\mathrm{PGF\text{-}gl.dim}\,R$ and \cite[Theorem~4.16]{ElM}, we deduce that $\Ggl R=\mathrm{PGF\text{-}gl.dim}\,R\le\Gwgl R+m+1$, as needed.
\end{proof}

\begin{remark}[An alternative proof of Corollary~\ref{cor4.4} for rings of cardinality at most $\alephm$]\label{rem4.6}
Assume that $|R|\le\alephm$. By a theorem of Simson \cite[Theorem]{Simson}, obtained independently by Gruson and Jensen \cite{GJ}, every flat module over a ring of cardinality at most $\alephm$ has projective dimension at most $m+1$. In other words, $\operatorname{splf}R\le m+1$, where $\operatorname{splf}R$ denotes the supremum of the projective dimensions of the flat $R$-modules. On the other hand, for every ring $R$ the supremum of the PGF-dimensions of the Gorenstein flat $R$-modules is equal to $\operatorname{splf}R$ (see \cite[Proposition~15]{DE}). Therefore, every Gorenstein flat $R$-module has PGF-dimension at most $m+1$. Now let $M$ be any $R$-module such that $\Gfd_RM=d<\infty$ and let $K_d$ be the $d$-th syzygy of a projective resolution of $M$. Since every ring is GF-closed (see \cite[Corollary~4.12]{SS}), the module $K_d$ is Gorenstein flat by \cite[Theorem~2.8]{Ben}, so that $\PGFdim_RK_d\le m+1$, and \cite[Lemma~1.2]{St0} gives $\PGFdim_RM\le\PGFdim_RK_d+d\le d+m+1$. This argument does not use Theorem~\ref{theoA} and imposes no condition on the module $M$. Taking the supremum over all $R$-modules, we deduce the inequality $\mathrm{PGF\text{-}gl.dim}\,R\le\Gwgl R+m+1$ of Corollary~\ref{cor4.5} for $|R|\le\alephm$.
\end{remark}

\section{Groups}\label{sec:groups}

This section is devoted to groups: we prove Theorem~\ref{theoB} and Corollary~\ref{corC}, we examine the sharpness of the bounds that they provide, and we derive from them bounds for the invariants of the group algebra.

\subsection{Proofs of Theorem~\texorpdfstring{\ref{theoB}}{B} and Corollary~\texorpdfstring{\ref{corC}}{C}}

Two inequalities frame the group-theoretic statements. Since PGF modules are Gorenstein projective \cite{SS}, every PGF-resolution of $k$ is in particular a Gorenstein projective resolution, whence $\Gcd_kG\le\PGFcd_kG$; and when $\sfli k<\infty$ we have $\Ghd_kG\le\Gcd_kG$ by \cite[Theorem~3.10]{KS}. The first of these two inequalities is in fact an equality whenever $\sfli k<\infty$, by \cite[Proposition 2.13]{St0}. The same equality holds for an arbitrary commutative ring $k$ whenever $\Ghd_kG<\infty$ and $|G|\le\alephm$: Theorem~\ref{theoB} then gives $\PGFcd_kG<\infty$, and $\Gpd_RM=\PGFdim_RM$ for every module of finite PGF-dimension, by \cite[Corollary~13(ii)]{DE}. Theorem~\ref{theoB} yields an \emph{upper} bound for $\PGFcd_kG$, which we now prove.

\begin{proof}[Proof of Theorem~\ref{theoB}]
We may assume that $\Ghd_kG<\infty$, since otherwise there is nothing to prove. Let $|G|\le\alephm$ and $\Ghd_kG=\Gfd_{kG}k=n<\infty$. We consider the (unnormalized) bar resolution $$\mathbf B:
\cdots\rightarrow B_2\xrightarrow{d_2}B_1\xrightarrow{d_1}B_0\xrightarrow{\varepsilon}k\rightarrow0,$$
in which $B_j$ is the free $kG$-module on the set $G^{j}$ of $j$-tuples of elements of $G$ and $\varepsilon$ is the augmentation; see, e.g., \cite[Chapter~I.5]{Brown}. Since $|G^j|=|G|^j\le\alephm^{\,j}=\alephm$ for every $j\ge1$, and $B_0=kG$ is free of rank one, each $B_j$ is free on at most $\alephm$ generators. Thus $\mathbf B$ is a projective resolution of the $kG$-module $k$ all of whose terms are $\alephm$-generated, and $\Gfd_{kG}k=n$. Corollary~\ref{cor4.3}, applied to the ring $kG$, the module $k$ and the exact sequence $B_{n+1}\to B_n\to\cdots\to B_0\to k\to0$, yields $\PGFcd_kG=\PGFdim_{kG}k\;\le\;n+m+1$. We conclude that $\Gcd_kG\le\PGFcd_kG\le\Ghd_kG+m+1$, as needed. For $m=0$, $G$ is a countable group and the inequalities $\Gcd_kG\le\PGFcd_kG\le\Ghd_kG+1$ hold.
\end{proof}

\begin{proof}[Proof of Corollary~\ref{corC}]
The first inequality, $\Ghd_kG\le\Gcd_kG$, is \cite[Theorem~3.10]{KS}. For the second, if $\Ghd_kG=\infty$ there is nothing to prove; and if $\Ghd_kG=n<\infty$, then Theorem~\ref{theoB} gives $\Gcd_kG\le\PGFcd_kG\le n+1=\Ghd_kG+1$. The two inequalities together show that $\Ghd_kG$ and $\Gcd_kG$ are finite simultaneously.
\end{proof}

Theorem~\ref{theoB} applies in particular to locally finite groups, for which the Gorenstein homological dimension vanishes.

\begin{corollary}\label{cor5.1}
Let $G$ be a locally finite group of cardinality at most $\alephm$ and let $k$ be a nonzero commutative ring. Then $\Gcd_kG\le\PGFcd_kG\le m+1$. If $G$ is countably infinite, then $\Gcd_kG=\PGFcd_kG=1$.
\end{corollary}

\begin{proof}
Since $G$ is locally finite, $\Ghd_kG=0$ (see \cite[Theorem~1.4]{KS}), and Theorem~\ref{theoB} gives the first assertion. If $G$ is infinite, then \cite[Corollary~2.3]{ET} yields $\Gcd_kG\ge1$, and the second assertion follows from the first for $m=0$.
\end{proof}

\subsection{Sharpness of the bounds}

The estimates obtained so far are best possible. We show in this subsection that both inequalities of Corollary~\ref{corC} are attained, the upper one in every degree and by countable groups (see Remark~\ref{rem5.2} and Proposition~\ref{prop5.3}), and that the error term $m+1$ of Theorem~\ref{theoB} cannot be improved for any value of $m$, already among locally finite abelian groups (see Proposition~\ref{prop5.4}).

\begin{remark}\label{rem5.2}
Both bounds in Corollary~\ref{corC} are attained, and so is the bound of Theorem~\ref{theoB} for $m=0$. The lower bound $\Gcd_kG=\Ghd_kG$ is realized by all groups of type $\mathrm{FP}_\infty$ over a commutative ring $k$ such that $\sfli k <\infty$ (see \cite[Theorem~B]{St}). For the upper bound, let $k$ be a nonzero commutative ring and $G$ be a countably infinite locally finite group. Then $\Ghd_kG=0$, since $G$ is locally finite (see \cite[Theorem~1.4]{KS}), whereas $\Gcd_kG=\PGFcd_kG=1$ by Corollary~\ref{cor5.1}. Consequently, \[
\Gcd_kG=\PGFcd_kG=1=\Ghd_kG+1 .
\]
The next result shows that the same phenomenon occurs in every degree.
\end{remark}

\begin{proposition}\label{prop5.3}
Let $k$ be a nonzero commutative ring with $\sfli k<\infty$ and let $n$ be a nonnegative integer. Then, there exists a countable group $G$ such that
\[
\Ghd_kG=n
\,\,\,\,\,\,\text{and}\,\,\,\,\,\,
\Gcd_kG=\PGFcd_kG=n+1=\Ghd_kG+1 .
\]
\end{proposition}

\begin{proof}
Let $A$ be a countably infinite locally finite group, set $N=\mathbb Z^n$ and $G=A\times N$, so that $G$ is countable. The torus $T^n=(S^1)^n$ has fundamental group $N$ and
contractible universal cover $\mathbb R^n$. It is therefore
an Eilenberg--MacLane complex of type $K(N,1)$
(see \cite[Example~1B.5]{Hatcher2002}). Equip each circle with the CW structure consisting of one vertex and one edge. The induced product CW structure on $T^n$ has $\binom{n}{j}$ cells of dimension $j$, corresponding to the choices of the $j$ factors that contribute an edge (see \cite[p.~8]{Hatcher2002} for the product construction). Lifting this CW structure to $\mathbb R^n$, the action of $N$ by deck transformations is cellular and free, with one orbit of cells for each cell of $T^n$. Hence, writing $C_j=C_j(\mathbb R^n;k)$, we have $C_j\cong (kN)^{\binom{n}{j}}$, for every $0\leq j\leq n$. Since $\mathbb R^n$ is contractible, the augmented cellular
chain complex
\[
\mathbf C:
0\rightarrow C_n
\rightarrow\cdots
\rightarrow C_1
\rightarrow C_0
\xrightarrow{\varepsilon} k
\rightarrow 0
\]
is exact, where $\varepsilon$ sends each vertex to $1\in k$.
By the standard cellular construction of free resolutions
(see \cite[Chapter~I, \S4]{Brown}), this is a resolution of the
trivial $kN$-module $k$ by finitely generated free $kN$-modules
of length $n$. Consequently, $N$ is of type
$\mathrm{FP}_{\infty}$ over $k$ and moreover
$\operatorname{cd}_k N\leq n$. Furthermore, $T^n$ is a closed, connected, orientable,
aspherical $n$-manifold. Its fundamental group $N$ is
therefore an orientable Poincar\'e duality group of
dimension $n$ (see \cite[Chapter~VIII, \S10, Example~1]{Brown}).
Poincar\'e duality thus yields natural isomorphisms
$H^i(N,M)\cong H_{n-i}(N,M)$ for every $0\leq i\leq n$
and every $\mathbb ZN$-module $M$ (see \cite[Chapter~VIII, Theorem~10.1 and p.~222]{Brown}).
In particular, this applies to $M=kN$, regarded as a
$\mathbb ZN$-module by restriction of scalars.

For a $kN$-module $M$, the group homology and cohomology
appearing here may equivalently be computed over $kN$.
Indeed, the integral augmented cellular complex of
$\mathbb R^n$ is split exact as a complex of abelian groups;
tensoring it with $k$ over $\mathbb Z$ therefore gives the
free $kN$-resolution $\mathbf C$ above. The usual
tensor--Hom adjunctions identify the resulting complexes
computing homology and cohomology with coefficients in $M$.
Since $kN$ is free over itself, we have
\[
H_j(N,kN)
\cong \operatorname{Tor}^{kN}_j(k,kN)
=
\begin{cases}
k, & j=0,\\
0, & j>0.
\end{cases}
\]
Combining this computation with Poincar\'e duality and the
bound $\operatorname{cd}_k N\leq n$, we obtain
\[
H^i(N,kN)\cong
\begin{cases}
0, & i\neq n,\\[1mm]
k, & i=n.
\end{cases}
\]
In particular, $H^i(N,kN)=0$ is projective for every $i<n$, and $H^n(N,kN)\cong k$ contains $k$ as a direct summand; these are the conditions on $N$ required in \cite[Theorem~4.5]{ET3}. Since $H^n(N,kN)\neq0$, we also have $\cd_k N=n$, and hence $\Gcd_k N=n$ by \cite[Proposition 2.27]{Hol}. We may therefore apply \cite[Theorem~4.5]{ET3} to the split extension $1\rightarrow N\rightarrow G\rightarrow A\rightarrow1$, with $Q=A$. Since $\Gcd_kA=1$ by Corollary~\ref{cor5.1}, we obtain $\Gcd_kG=\Gcd_k N+\Gcd_kA=n+1$. Finally, we let $G=A\times\mathbb Z^n$ act on $\mathbb R^n$ through the projection $G\rightarrow\mathbb Z^n$, the subgroup $A$ acting trivially. This is a cellular action of $G$ on a contractible $G$-CW-complex of dimension $n$, all of whose cell stabilizers are equal to $A$. Consequently, $\mathbf C$ becomes an exact sequence of $kG$-modules whose terms are direct sums of copies of $k[G/A]=\mathrm{Ind}^G_Ak$ (see \cite[Lemma~3.1]{Bis}). Since $A$ is locally finite we have $\Ghd_kA=0$ (see \cite[Theorem~1.4]{KS}), i.e.\ $k$ is a Gorenstein flat $kA$-module. Induction from $A$ to $G$ preserves Gorenstein flat modules (see \cite[Lemma~2.9(i)]{KS}), and hence $\mathrm{Ind}^G_Ak$ is a Gorenstein flat $kG$-module. Since the class ${\tt GFlat}(kG)$ is closed under finite direct sums, it follows that each term
$C_j\cong(\mathrm{Ind}^G_Ak)^{\binom{n}{j}}$ of $\mathbf C$ is a Gorenstein flat $kG$-module. It follows that $\Ghd_kG\le n$, so that $n+1=\Gcd_kG\le\Ghd_kG+1\le n+1$ by Corollary~\ref{corC}. We conclude that $\Ghd_kG=n$ and $\Gcd_kG=\Ghd_kG+1=n+1$. Finally, Theorem~\ref{theoB} gives $n+1=\Gcd_kG\le\PGFcd_kG\le\Ghd_kG+1=n+1$, and hence $\PGFcd_kG=n+1$ as well.
\end{proof}


\begin{proposition}\label{prop5.4}
Let $m\geq0$ be a nonnegative integer.
\begin{enumerate}
\item[(i)] The locally finite abelian group $G_m=\bigoplus_{\alpha<\aleph_m}C_2$, where $C_2$ denotes the cyclic group of order two, satisfies $\Ghd_{\mathbb Z}G_m=0$ and $\Gcd_{\mathbb Z}G_m=\PGFcd_{\mathbb Z}G_m=m+1$. In particular, the bound of Theorem~\ref{theoB} is sharp for every finite $m$, already among locally finite abelian groups.
\item[(ii)] If $m=1$, then every locally finite group $G$ with $|G|=\aleph_1$ satisfies $\Gcd_{\mathbb Z}G=\PGFcd_{\mathbb Z}G=2=\Ghd_{\mathbb Z}G+2$.
\end{enumerate}\end{proposition}

\begin{proof}
(i) Let $G$ be a locally finite group with
$|G|=\aleph_m$. The group algebra $\mathbb{Q}G$ is a
directed union of the semisimple algebras $\mathbb{Q}H$,
where $H$ ranges over the finite subgroups of $G$.
Since von Neumann regularity is preserved under directed unions, $\mathbb{Q}G$ is
von Neumann regular (see also
\cite[Theorem~3 and the subsequent Remark~(b), p.~659]{Connell1963}).
Consequently,
$\wgl(\mathbb{Q}G)=0$
by \cite[Theorem~4.2.9]{Weibel1994}.
Since $|\mathbb{Q}G|=\aleph_m$, Osofsky's bound
\cite[Corollary~1.4]{Osofsky1968} gives $\gldim(\mathbb{Q}G)\leq \wgl(\mathbb{Q}G)+m+1
=m+1.$
In particular, the trivial $\mathbb{Q}G$-module $\mathbb{Q}$
has finite projective dimension. Since Gorenstein projective
dimension agrees with projective dimension whenever the latter
is finite (see \cite[Proposition~2.27]{Hol}), we obtain that $\Gcd_{\mathbb{Q}}G=\cd_{\mathbb{Q}}G.$
Using \cite[Corollary~4.2]{Ren} together with Corollary~\ref{cor5.1}, we deduce that
\begin{equation}\label{eq6}
	\cd_{\mathbb{Q}}G
=\Gcd_{\mathbb{Q}}G
\leq\Gcd_{\mathbb{Z}}G
\leq\PGFcd_{\mathbb{Z}}G
\leq m+1.
\end{equation}
We now specialize this to $G=G_m$, which is a locally finite abelian group of cardinality $\aleph_m$. By Chen's theorem, in the form recorded in
\cite[Corollary~6.10(2)]{DKLT2002}, we have
$\cd_{\mathbb{Q}}G_m=m+1.$
Applying the preceding comparison \eqref{eq6} to $G_m$, we deduce that
$\Gcd_{\mathbb{Z}}G_m
=\PGFcd_{\mathbb{Z}}G_m
=m+1$.
Moreover, local finiteness implies
$\Ghd_{\mathbb{Z}}G_m=0$ (see \cite[Theorem~1.4]{KS}). Hence
$\PGFcd_{\mathbb{Z}}G_m
=\Ghd_{\mathbb{Z}}G_m+m+1$, as needed.

(ii) Let $m=1$ and $G$ be a locally finite group of cardinality $\aleph_1$. Then
\cite[Theorem~5.4]{DKLT2002} yields $H^2(G,\mathbb{Q}G)\neq 0$,
and hence $\cd_{\mathbb{Q}}G\geq 2$.
Together with the upper bound established above, this yields
$\cd_{\mathbb{Q}}G=2$.
Since this dimension is finite, we also have
$\Gcd_{\mathbb{Q}}G=2$.
The preceding comparison therefore gives
$2=\Gcd_{\mathbb{Q}}G
\leq\Gcd_{\mathbb{Z}}G
\leq\PGFcd_{\mathbb{Z}}G
\leq 2.$
Consequently,
$\Gcd_{\mathbb{Z}}G
=\PGFcd_{\mathbb{Z}}G
=2.$
Since $G$ is locally finite we also have $\Ghd_{\mathbb Z}G=0$ (see \cite[Theorem~1.4]{KS}) and (ii) follows.
\end{proof}

\subsection{Applications to group algebras} We finally specialize the bound of Corollary~\ref{cor4.5} to group algebras. Since $kG\cong(kG)^{\mathrm{op}}$, the Gorenstein weak global dimension of $kG$ is the Gedrich--Gruenberg invariant $\sfli(kG)$, which is in turn bounded in terms of $\Ghd_kG$ and $\sfli k$ (see \cite[Corollary~6.2]{KS}); the two corollaries below combine these facts.

The hypothesis needed for Corollary~\ref{cor4.5} concerns the number of generators of the left ideals of $kG$. The following lemma shows that it is guaranteed by the corresponding condition on the ideals of $k$ alone.

\begin{lemma}\label{lem5.5} Let $k$ be a commutative ring, $G$ be a group and $m$ be a nonnegative integer such that $|G|\le\alephm$ and every ideal of $k$ is $\alephm$-generated. Then, every left ideal of $kG$ is $\alephm$-generated, and indeed $\alephm$-generated as a $k$-module. In particular, the conclusion holds whenever
$|G|\leq\aleph_m$ and $|k|\leq\aleph_m$.\end{lemma}

\begin{proof}Let $\kappa=\aleph_m$ and $I$ be a left ideal of $kG$.
We shall prove that $I$ can be generated by at most $\kappa$ elements even as a $k$-module. The group algebra $kG$ is a free $k$-module with basis $G$.
Choose a well-ordering $G=\{g_\alpha:\alpha<\lambda\}$,
where $\lambda$ is an ordinal such that $|\lambda|=|G|\leq\kappa$.
For every $\alpha\leq\lambda$, we define the $k$-submodules $F_\alpha=\bigoplus_{\beta<\alpha}kg_\beta$ and $I_\alpha=I\cap F_\alpha$ of $kG$, where $I_0=0$, $I_\lambda=I$ and $I_\alpha\subseteq I_{\alpha+1}$ for every $\alpha<\lambda$. For each $\alpha<\lambda$, we consider the $k$-linear map
$\varphi_\alpha:I_{\alpha+1}\rightarrow k$ that assigns to an element its coefficient at $g_\alpha$. Since
$F_{\alpha+1}=F_\alpha\oplus kg_\alpha$,
an element of $I_{\alpha+1}$ lies in the kernel of
$\varphi_\alpha$ precisely when it belongs to $I_\alpha$.
Thus $\ker\varphi_\alpha=I_\alpha$. Moreover, the image
$J_\alpha=\operatorname{im}\varphi_\alpha$
is a $k$-submodule of $k$, hence an ideal of $k$.
The first isomorphism theorem therefore gives
$I_{\alpha+1}/I_\alpha\cong J_\alpha$
as $k$-modules. By hypothesis, $J_\alpha$ has a generating set of
cardinality at most $\kappa$. Choosing preimages of these generators under
$\varphi_\alpha$, we obtain a subset
$S_\alpha\subseteq I_{\alpha+1}$ with
$|S_\alpha|\leq\kappa$ whose image generates
$I_{\alpha+1}/I_\alpha$.
Equivalently, $I_{\alpha+1}=I_\alpha+\langle S_\alpha\rangle_k$,
where $\langle S_\alpha\rangle_k$ denotes the
$k$-submodule generated by $S_\alpha$. We set $S=\bigcup_{\alpha<\lambda}S_\alpha$.
Since $\kappa$ is infinite and $|\lambda|\leq\kappa$,
we deduce that $|S|\leq |\lambda|\cdot\kappa
\leq \kappa\cdot\kappa=\kappa$.
It remains to show that $S$ generates $I$ as a $k$-module. Using transfinite induction we will show that $I_\alpha=\left\langle\bigcup_{\beta<\alpha}S_\beta\right\rangle_k$, for every $\alpha\leq\lambda$. The inclusion from right to left follows directly from the inclusions $S_\beta\subseteq I_{\beta+1}\subseteq I_\alpha$ for every $\beta<\alpha$.
For the reverse inclusion, the case $\alpha=0$ is immediate. We suppose that the assertion holds for $\alpha$
and let $x\in I_{\alpha+1}$.
Since the image of $S_\alpha$ generates
$I_{\alpha+1}/I_\alpha$, there exist
$s_1,\ldots,s_t\in S_\alpha$ and
$a_1,\ldots,a_t\in k$ such that $x-\sum_{r=1}^{t}a_rs_r\in I_\alpha$.
By the induction hypothesis, this difference is a finite
$k$-linear combination of elements of
$\bigcup_{\beta<\alpha}S_\beta$.
Hence $x$ belongs to the $k$-submodule generated by
$\bigcup_{\beta<\alpha+1}S_\beta$. Now let $\delta\leq\lambda$ be a nonzero limit ordinal and we assume that the assertion holds for all
$\alpha<\delta$. Every element of $kG$ has finite support. Consequently, for each $x\in I_\delta$, there exists $\alpha<\delta$ such that $x\in F_\alpha$, and therefore $x\in I_\alpha$.
Thus $I_\delta=\bigcup_{\alpha<\delta}I_\alpha$, which is a $k$-submodule of $kG$, being the union of an increasing chain of $k$-submodules, and the required assertion follows from the induction
hypothesis. Taking $\alpha=\lambda$, we conclude that
$I=I_\lambda=\langle S\rangle_k$. Finally, since $S\subseteq I$ and $I$ is a left ideal, $I=\sum_{s\in S}ks
\subseteq\sum_{s\in S}(kG)s
\subseteq I$.
Hence $S$ also generates $I$ as a left ideal of $kG$,
and $|S|\leq\kappa=\aleph_m$. For the final assertion, if $|k|\leq\kappa$, then every ideal of $k$ has cardinality at most $\kappa$ and is generated by its own elements.
\end{proof}

\begin{corollary}\label{cor5.6}
Let $k$ be a commutative ring, $G$ be a group, and $m$ be a nonnegative integer such that $|G|\le\alephm$ and every ideal of $k$ is $\alephm$-generated. Then,
\[
\Ggl(kG)\;\le\;\Gwgl(kG)+m+1\;=\;\sfli(kG)+m+1 .
\]
\end{corollary}

\begin{proof} By Lemma~\ref{lem5.5}, every left ideal of $kG$ is an $\alephm$-generated $kG$-module, so Corollary~\ref{cor4.5} applies to $R=kG$ and yields $\Ggl(kG)\le\Gwgl(kG)+m+1$. Finally, $kG\cong(kG)^{\mathrm{op}}$ and hence $\Gwgl(kG)=\sfli(kG)$ (see Section~\ref{sec:prelim}).
\end{proof}

Combined with the estimate of $\sfli(kG)$ obtained in \cite{KS}, Corollary~\ref{cor5.6} yields a bound on the Gedrich--Gruenberg invariants $\spli(kG)$ and $\silp(kG)$ of the group algebra in terms of the Gorenstein homological dimension of the group itself.

\begin{corollary}\label{cor5.7}
Let $k$ be a commutative ring, $G$ be a group, and $m$ be a nonnegative integer such that $|G|\le\alephm$ and every ideal of $k$ is $\alephm$-generated. Then, \[\silp(kG)\le \spli(kG)\le\sfli(kG)+m+1\le\Ghd_kG+\sfli k+m+1 .\]
In particular, if $k$ is a field, then $\silp(kG)\le\spli(kG)\le\Ghd_kG+m+1$.
\end{corollary}

\begin{proof}
The right inequality follows from the estimate $\sfli(kG)\le\Ghd_kG+\sfli k$ of \cite[Corollary~6.2]{KS}. Moreover, the left inequality holds by \cite[Corollary~24]{DE}, since $kG\cong(kG)^{\mathrm{op}}$. For the middle inequality, we may assume that $\sfli(kG)<\infty$, since otherwise there is nothing to prove. Then, Corollary~\ref{cor5.6} yields the finiteness of $\Ggl(kG)$. Consequently, $\Ggl(kG)=\spli(kG)$ (see \cite[Theorem~4.1]{Emm}) and the middle inequality follows from Corollary~\ref{cor5.6}. Finally, if $k$ is a field, then $\sfli k=0$, and the two-sided estimate $\max\{\Ghd_kG,\sfli k\}\le\sfli(kG)\le\Ghd_kG+\sfli k$ of \cite[Corollary~6.2]{KS} gives $\sfli(kG)=\Ghd_kG$ and the last inequality follows.
\end{proof}

The estimates above yield a finiteness criterion: for a group of cardinality at most $\alephm$, the Gedrich--Gruenberg invariants of the group algebra are finite precisely when the Gorenstein dimensions of the group are.

\begin{corollary}\label{cor5.8} Let $k$ be a commutative ring with $\sfli k<\infty$, $G$ be a group, and $m$ be a nonnegative integer such that $|G|\le\alephm$ and every ideal of $k$ is $\alephm$-generated. Then, the following conditions are equivalent:
\begin{enumerate}
\item[(i)] $\Ghd_kG<\infty$;
\item[(ii)] $\Gcd_kG=\PGFcd_kG<\infty$;
\item[(iii)] $\sfli(kG)<\infty$;
\item[(iv)] $\spli(kG)<\infty$;
\item[(v)] $\Gwgl(kG)<\infty$;
\item[(vi)] $\Ggl(kG)<\infty$.
\end{enumerate}
\end{corollary}

\begin{proof} Since $\sfli k<\infty$, we have $\Ghd_kG\le\Gcd_kG=\PGFcd_kG\le\Ghd_kG+m+1$, by \cite[Theorem~3.10]{KS}, \cite[Proposition 2.13]{St0} and Theorem~\ref{theoB}. Consequently, the conditions (i) and (ii) are equivalent. Since $\sfli(kG)\le\spli(kG)$, invoking \cite[Corollary~6.2]{KS} and Corollary~\ref{cor5.7}, we obtain a chain of inequalities
\[
\Ghd_kG\;\le\;\sfli(kG)\;\le\;\spli(kG)\;\le\;\sfli(kG)+m+1\;\le\;\Ghd_kG+\sfli k+m+1 ,
\]
whose outer terms are finite simultaneously, since $\sfli k<\infty$; this proves the equivalence of (i), (iii) and (iv). Moreover, Corollary~\ref{cor5.6} gives $\Ggl(kG)\le\sfli(kG)+m+1$, so that (iii) implies (vi). Conversely, if $\Ggl(kG)<\infty$, then $\spli(kG)=\silp(kG)=\Ggl(kG)<\infty$ (see \cite[Theorem~4.1]{Emm}), which shows that (vi) implies (iv). Finally, since $kG\cong(kG)^{\mathrm{op}}$, we have $\Gwgl(kG)=\sfli(kG)$ (see Section~\ref{sec:prelim}), so that (iii) and (v) are equivalent. \end{proof}

\begin{remark}\label{rem5.9}
In contrast to Theorem~\ref{theoB}, Corollaries~\ref{cor5.6} and \ref{cor5.7} require a hypothesis on the ideals of $k$, as well as a bound on the cardinality of $G$. The bar resolution provides a projective resolution of the trivial module $k$ with $\alephm$-generated terms regardless of the cardinality of $k$, whereas the left ideals of $kG$ need not be $\alephm$-generated when the ideals of $k$ are not.\end{remark}

\appendix

\section{A characterization of the PGF-dimension}\label{sec:appendix}

The proof of Lemma~\ref{lem4.1}, which constitutes the main tool of this paper, rests on a single principle: a module which is periodic modulo a module of projective dimension at most $n$, and whose $(n+1)$-st $\Tor$ against the injective modules vanishes, has PGF-dimension at most $n$. We prove this here, in the form of a characterization of the modules of PGF-dimension at most $n$ as the direct summands of the periodic modules just described (see Theorem~\ref{theoA.5}). The case $n=0$ recovers the description of the PGF modules as the direct summands of the strongly PGF ones, which follows from \cite[Lemma~4.1]{SS} together with the closure of $\PGF(R)$ under direct summands. Everything is deduced from the properties of the PGF-dimension collected in Proposition~\ref{prop2.3}. Throughout, $R$ is an arbitrary ring and $n$ is a nonnegative integer.

\begin{definition}\label{defA.1}
An $R$-module $M$ is \emph{strongly $n$-PGF} if there exists a short exact sequence of $R$-modules
\[
0\rightarrow M\rightarrow F\rightarrow M\rightarrow0
\]
with $\pd_RF\le n$ and $\Tor^R_{n+1}(I,M)=0$ for every injective right $R$-module $I$.
\end{definition}

For $n=0$ the condition reads: there is a short exact sequence $0\to M\to F\to M\to0$ with $F$ projective which remains exact after applying $I\otimes_R-$ for every injective right $R$-module $I$. We call such a module \emph{strongly PGF}, in analogy with the strongly Gorenstein projective modules of Bennis--Mahdou \cite{BM}. On the other hand, every module $M$ with $\pd_RM\le n$ is strongly $n$-PGF, as the split sequence $0\to M\to M\oplus M\to M\to0$ shows.

\begin{proposition}\label{propA.2}
For an $R$-module $M$ the following assertions are equivalent:
\begin{enumerate}
\item[(i)] $M$ is strongly $n$-PGF;
\item[(ii)] there is a short exact sequence $0\to M\to F\to M\to0$ with $\pd_RF\le n$ such that $\Tor^R_i(I,M)=0$ for every $i>n$ and every injective right $R$-module $I$.
\end{enumerate}
\end{proposition}

\begin{proof}
Only (i)$\Rightarrow$(ii) requires an argument, and the sequence provided by Definition~\ref{defA.1} will serve. Let $I$ be an injective right $R$-module. Since $\pd_RF\le n$, we have $\Tor^R_i(I,F)=0$ for every $i>n$, so that the part $\Tor^R_{i+1}(I,F)\rightarrow\Tor^R_{i+1}(I,M)\rightarrow\Tor^R_i(I,M)\rightarrow\Tor^R_i(I,F)$
of the long exact $\Tor$ sequence associated with the short exact sequence $0\to M\to F\to M\to0$ has vanishing outer terms for every $i>n$. Consequently, $\Tor^R_{i+1}(I,M)\cong\Tor^R_i(I,M)$ for every $i>n$, and since $\Tor^R_{n+1}(I,M)=0$, we deduce that $\Tor^R_i(I,M)=0$ for every $i\ge n+1$.
\end{proof}

\begin{lemma}\label{lemA.3}
Every strongly PGF module is PGF.
\end{lemma}

\begin{proof}
Let $N$ be strongly PGF and let $0\rightarrow N\xrightarrow{\alpha}Q\xrightarrow{\beta}N\rightarrow 0$ be a short exact sequence with $Q$ projective and $\Tor^R_1(I,N)=0$ for every injective right $R$-module $I$. Splicing this sequence with itself in both directions yields the acyclic complex of projective modules
\[
\mathbf Q:\quad\cdots\rightarrow Q\xrightarrow{\alpha\beta}Q\xrightarrow{\alpha\beta}Q\rightarrow\cdots,
\]
all of whose cycles are isomorphic to $N$. For an injective right $R$-module $I$ the vanishing $\Tor^R_1(I,N)=0$ makes $0\to I\otimes_RN\to I\otimes_RQ\to I\otimes_RN\to0$ exact, and splicing these back together shows that $I\otimes_R\mathbf Q$ is acyclic. Thus $N$ is a cycle of an acyclic complex of projectives which stays acyclic under $I\otimes_R-$, i.e.\ $N$ is PGF.
\end{proof}

\begin{proposition}\label{propA.4}
Let $M$ be a strongly $n$-PGF $R$-module.
\begin{enumerate}
\item[(i)] If $n\ge1$ and $0\to N\to P_{n-1}\to\cdots\to P_0\to M\to0$ is an exact sequence with $P_0,\dots,P_{n-1}$ projective, then $N$ is strongly PGF and consequently $\PGFdim_RM\le n$; this last inequality also holds for $n=0$, by Lemma~\ref{lemA.3}.\item[(ii)] If moreover $0\to M\to F\to M\to0$ is a short exact sequence with $\pd_RF<\infty$, then $\PGFdim_RM=\pd_RF$. In particular, $M$ is strongly $q$-PGF for $q:=\pd_RF$.
\end{enumerate}
\end{proposition}

\begin{proof}
(i) Choose a short exact sequence $0\to M\to F\to M\to0$ as in Definition~\ref{defA.1}, so that $\pd_RF\le n$ and $\Tor^R_{n+1}(I,M)=0$ for every injective right $R$-module $I$. The given sequence is a truncated projective resolution of $M$; applying the horseshoe lemma to $0\to M\to F\to M\to0$ with this resolution at both ends produces a commutative diagram with exact rows and columns
\[
\begin{array}{ccccccccccccc}
&   &0& &0& & & &0& &0& & \\&   &\downarrow& &\downarrow& & & &\downarrow& &\downarrow& & \\
0&\to&N&\to&P_{n-1}&\to&\cdots&\to&P_0&\to&M&\to&0\\
 &   &\downarrow& &\downarrow& & & &\downarrow& &\downarrow& & \\
0&\to&Q&\to&P_{n-1}\oplus P_{n-1}&\to&\cdots&\to&P_0\oplus P_0&\to&F&\to&0\\
 &   &\downarrow& &\downarrow& & & &\downarrow& &\downarrow& & \\
0&\to&N&\to&P_{n-1}&\to&\cdots&\to&P_0&\to&M&\to&0\\
&   &\downarrow& &\downarrow& & & &\downarrow& &\downarrow& & \\
&   &0& &0& & & &0& &0& & \\
\end{array}
\]
whose middle row is a truncated projective resolution of $F$ and whose left-hand column is a short exact sequence $0\to N\to Q\to N\to0$. As $\pd_RF\le n$, the $R$-module $Q$ is projective. Moreover, dimension shifting along the top row gives $\Tor^R_1(I,N)\cong\Tor^R_{n+1}(I,M)=0$ for every injective right $R$-module $I$, and therefore $N$ is strongly PGF. By Lemma~\ref{lemA.3} the module $N$ is then PGF, so that $$0\to N\to P_{n-1}\to\cdots\to P_0\to M\to0$$ is a PGF-resolution of $M$ of length $n$ and $\PGFdim_RM\le n$. For $n=0$ the module $M$ is itself strongly PGF, hence PGF by Lemma~\ref{lemA.3}, and therefore $\PGFdim_RM=0$. 

(ii) Let $q=\pd_RF<\infty$. We consider a truncated projective resolution of the module $M$ of length $n$,
\[
0\rightarrow N\rightarrow P_{n-1}\rightarrow\cdots\rightarrow P_0\rightarrow M\rightarrow0 .
\]
Using part (i), which we have already proved, we obtain that $N$ is strongly PGF. By Proposition~\ref{propA.2}(ii), there exists a short exact sequence $0\rightarrow N\rightarrow P\rightarrow N\rightarrow0$ such that $P$ is projective and $\Tor^R_i(I,N)=0$ for every $i>0$ and every injective right $R$-module $I$. Let $I$ be an injective right $R$-module. The short exact sequence $0\rightarrow M\rightarrow F\rightarrow M\rightarrow0$ induces a long exact sequence of the form
\[
\cdots\rightarrow\Tor^R_{i+1}(I,F)\rightarrow\Tor^R_{i+1}(I,M)\rightarrow\Tor^R_i(I,M)\rightarrow\Tor^R_i(I,F)\rightarrow\cdots,
\]
implying that $\Tor^R_{i+1}(I,M)\cong\Tor^R_i(I,M)$ for every $i>q$. Thus,
\[
\Tor^R_i(I,M)\cong\Tor^R_{n+i}(I,M)\cong\Tor^R_i(I,N)=0
\]
for every $i>q$, the second isomorphism being given by dimension shifting along the truncated projective resolution above. By Proposition~\ref{propA.2}(ii), we conclude that $M$ is strongly $q$-PGF. It remains to prove that $\PGFdim_RM=q$. By part (i) we have $\PGFdim_RM\le q$. Since $\pd_RF=q<\infty$, Proposition~\ref{prop2.3}(v) implies that $\PGFdim_RF=\pd_RF=q$. Using Proposition~\ref{prop2.3}(iv)(a) and the short exact sequence $0\rightarrow M\rightarrow F\rightarrow M\rightarrow0$, we get the inequality $q=\PGFdim_RF\le\PGFdim_RM$. We conclude that $\PGFdim_RM=q$.
\end{proof}

We can now prove the characterization of the modules of finite PGF-dimension announced above.

\begin{theorem}\label{theoA.5}
Let $M$ be an $R$-module and $n\ge0$ an integer. Then $\PGFdim_RM\le n$ if and only if $M$ is a direct summand of a strongly $n$-PGF module.
\end{theorem}

\begin{proof}
Assume first that $M$ is a direct summand of a strongly $n$-PGF module $N$. Then $\PGFdim_RM\le\PGFdim_RN$ by Proposition~\ref{prop2.3}(iii), while $\PGFdim_RN\le n$ by Proposition~\ref{propA.4}(i). Consequently, $\PGFdim_RM\le n$. Conversely, suppose that $\PGFdim_RM\le n$. By Proposition~\ref{prop2.3}(vi) there exists a short exact sequence
\begin{equation}\label{eq:app-up}
0\rightarrow M\rightarrow D\rightarrow G^0\rightarrow0
\end{equation}
with $G^0$ PGF and $\pd_RD\le n$. Since $G^0$ is a cycle of an acyclic complex of projectives which remains acyclic under $I\otimes_R-$, its right-hand half provides an exact sequence $0\to G^0\to P^0\to P^1\to\cdots$ with all $P^i$ projective, whose images $G^{i+1}:=\im(P^i\to P^{i+1})$ are again PGF; thus we obtain short exact sequences
\begin{equation}\label{eq:app-co}
0\rightarrow G^i\rightarrow P^i\rightarrow G^{i+1}\rightarrow0,
\qquad i\ge0,
\end{equation}
with $\PGFdim_RG^i=0$ for every $i\ge0$. In the other direction, choose a projective resolution $\cdots\to P_1\to P_0\to M\to0$ of $M$ and let $G_0=\ker(P_0\to M)$ and $G_j=\ker(P_j\to P_{j-1})$ for $j\ge1$, so that the sequences
\begin{equation}\label{eq:app-down}
0\rightarrow G_0\rightarrow P_0\rightarrow M\rightarrow0,
\qquad
0\rightarrow G_j\rightarrow P_j\rightarrow G_{j-1}\rightarrow0\ \ (j\ge1)
\end{equation}
are exact. The second inequality of Proposition~\ref{prop2.3}(iv), applied to the first of these sequences, gives $\PGFdim_RG_0\le\max\{\PGFdim_RP_0,\PGFdim_RM-1\}\le n$, and an induction on $j$ using the remaining ones gives $\PGFdim_RG_j\le n$ for every $j\ge0$. We now add up the short exact sequences \eqref{eq:app-co}, \eqref{eq:app-up} and \eqref{eq:app-down}. Setting
\[
N=\Bigl(\bigoplus_{i\ge0}G^i\Bigr)\oplus M\oplus\Bigl(\bigoplus_{j\ge0}G_j\Bigr),
\qquad
Q=\Bigl(\bigoplus_{i\ge0}P^i\Bigr)\oplus D\oplus\Bigl(\bigoplus_{j\ge0}P_j\Bigr),
\]
the left-hand terms of these sequences add up to $N$, the middle terms to $Q$, and the right-hand terms to
\[
\Bigl(\bigoplus_{i\ge1}G^i\Bigr)\oplus G^0\oplus M\oplus\Bigl(\bigoplus_{j\ge0}G_j\Bigr)=N .
\]
We therefore obtain a short exact sequence $0\rightarrow N\rightarrow Q\rightarrow N\rightarrow0.$
All the summands of $Q$ except $D$ are projective, so $\pd_RQ=\pd_RD\le n$. Furthermore, $\PGFdim_RN\le n$ by Proposition~\ref{prop2.3}(iii), since each summand of $N$ has PGF-dimension at most $n$; hence the $n$-th syzygy $\Omega^nN$ of a projective resolution of $N$ is PGF by Proposition~\ref{prop2.3}(ii), in particular Gorenstein flat, and a dimension shifting argument gives $\Tor^R_{n+1}(I,N)\cong\Tor^R_1(I,\Omega^nN)=0$ for every injective right $R$-module $I$. Thus $N$ is strongly $n$-PGF, and $M$ is one of its direct summands.
\end{proof}


\begin{remark}
Let $M$ be an $R$-module with
$\PGFdim_RM=n<\infty$.
By Theorem~\ref{theoA.5}, there exists a strongly $n$-PGF module
$N$ having $M$ as a direct summand. Consequently,
$n=\PGFdim_RM
\leq \PGFdim_RN
\leq n$,
where the first inequality follows from Proposition~\ref{prop2.3}(iii)
and the second one from
Proposition~\ref{propA.4}(i). Thus, $\PGFdim_RN=n$. By the definition of a strongly $n$-PGF module,
there exists a short exact sequence
$0\rightarrow N\rightarrow Q\rightarrow N
\rightarrow 0$
with $\pd_RQ\leq n$.
Proposition~\ref{propA.4}(ii) then yields
$\pd_RQ=\PGFdim_RN=\PGFdim_RM=n$.
Hence, the finite PGF-dimension of $M$ is realized
as the projective dimension of the middle term
of a periodic short exact sequence for a module
containing $M$ as a direct summand.
This description is analogous to the characterization
of Gorenstein projective modules as direct summands
of strongly Gorenstein projective modules
due to Bennis--Mahdou~\cite{BM}.
\end{remark}




\end{document}